\documentclass[final,reqno]{amsart}

\usepackage{fullpage}
\usepackage[foot]{amsaddr}
\usepackage{mathtools,amssymb}
\usepackage{mathrsfs,euscript}
\usepackage{braket,slashed}
\usepackage[only llbracket,rrbracket,longmapsfrom]{stmaryrd}
\usepackage{tikz}
\usetikzlibrary{cd}
\usepackage{showkeys,comment}
\usepackage[l2tabu,orthodox]{nag}
\usepackage[all, warning]{onlyamsmath}
\usepackage[bookmarks=false,draft=false,colorlinks]{hyperref}

\numberwithin{equation}{section}
\allowdisplaybreaks

{\par \vspace{\baselineskip}%
 \noindent \textbf{Acknowledgements.}}%
{\par \vspace{\baselineskip}}

\usepackage{enumitem}
\setenumerate{label=(\arabic*),nosep}
\setitemize{nosep}
\newlist{clist}{enumerate}{1}
\setlist*[clist]{label=(\roman*), nosep}

\usepackage{cleveref}
\crefname{thm}{Theorem}{Theorems}
\crefname{dfn}{Definition}{Definitions}
\crefname{prp}{Proposition}{Propositions}
\crefname{lem}{Lemma}{Lemmas}
\crefname{cor}{Corollary}{Corollaries}
\crefname{clm}{Claim}{Claims}
\crefname{fct}{Fact}{Facts}
\crefname{rmk}{Remark}{Remarks}
\crefname{eg}{Example}{Examples}
\crefname{figure}{Figure}{Figures}
\crefname{table}{Table}{Tables}
\crefname{section}{\S\!}{\S\S\!}
\crefname{subsection}{\S\!}{\S\S\!}
\crefname{subsubsection}{\S\!}{\S\S\!}
\crefname{appendix}{Appendix}{Appendices}
\crefname{equation}{}{}

\usepackage{amsthm}
\theoremstyle{definition}
\newtheorem{thm}{Theorem}[section]
\newtheorem{dfn}[thm]{Definition}
\newtheorem{prp}[thm]{Proposition}
\newtheorem{lem}[thm]{Lemma}
\newtheorem{cor}[thm]{Corollary}

\newtheorem{fct}[thm]{Fact}
\newtheorem{rmk}[thm]{Remark}
\newtheorem{eg}[thm]{Example}
\newtheorem*{rmk*}{Remark}

\newcommand{\pd}{\partial}

\newcommand{\vp}{\varpi}
\newcommand{\ol}{\overline}
\newcommand{\ul}{\underline}
\newcommand{\wh}{\widehat}

\newcommand{\ceq}{\coloneqq} 
\newcommand{\sd}{\slashed{\pd}} 

\newcommand{\xr}{\xrightarrow}
\newcommand{\xrr}[1]{\xrightarrow{\ #1 \ }{}}
\newcommand{\inj}{\hookrightarrow}
\newcommand{\srj}{\twoheadrightarrow}
\newcommand{\sr}[1]{\overset{#1}{\srj}}
\newcommand{\lto}{\longrightarrow}

\newcommand{\sto}{\xr{\sim}}
\newcommand{\lsto}{\xrr{\sim}}
\newcommand{\linj}{\lhook\joinrel\longrightarrow}

\newcommand{\mto}{\mapsto}
\newcommand{\lmto}{\longmapsto}

\newcommand{\sqr}{\sqrt{\ }}
\newcommand{\ev}{\ol{0}}
\newcommand{\od}{\ol{1}}
\newcommand{\Zt}{\bbZ/2\bbZ}
\newcommand{\ch}{\textup{ch}}
\newcommand{\qc}{\textup{qc}}
\newcommand{\tsc}{\textup{sc}}
\newcommand{\tred}{\textup{red}}
\newcommand{\even}{\textup{even}}
\newcommand{\tprd}{{\textstyle\prod}}

\newcommand{\bbA}{\mathbb{A}}
\newcommand{\bbC}{\mathbb{C}}
\newcommand{\bbN}{\mathbb{N}}
\newcommand{\bbZ}{\mathbb{Z}}
\newcommand{\clJ}{\mathcal{J}}
\newcommand{\clL}{\mathcal{L}}
\newcommand{\clS}{\mathcal{S}}
\newcommand{\shA}{\EuScript{A}}
\newcommand{\shD}{\EuScript{D}}
\newcommand{\shE}{\EuScript{E}}
\newcommand{\shF}{\EuScript{F}}
\newcommand{\shI}{\EuScript{I}}
\newcommand{\shJ}{\EuScript{J}}
\newcommand{\shK}{\EuScript{K}}
\newcommand{\shL}{\EuScript{L}}
\newcommand{\shM}{\EuScript{M}}
\newcommand{\shN}{\EuScript{N}}
\newcommand{\shO}{\EuScript{O}}
\newcommand{\shS}{\EuScript{S}}
\newcommand{\frg}{\mathfrak{g}}
\newcommand{\frk}{\mathfrak{k}}
\newcommand{\frm}{\mathfrak{m}}

\newcommand{\JS}{\clJ\clS}
\newcommand{\LS}{\clL\clS}

\newcommand{\catC}{\mathsf{C}}
\newcommand{\cSet}{\mathsf{Set}}
\newcommand{\SSp}{\mathsf{SSp}}
\newcommand{\SSch}{\mathsf{SSch}}

\newcommand{\JC}[1]{\clJ\clS_{C^{#1}}}
\newcommand{\LC}[1]{\clL\clS_{C^{#1}}}
\newcommand{\abs}[1]{\left| #1 \right|}
\newcommand{\dbr}[1]{\llbracket #1 \rrbracket} 
\newcommand{\dpr}[1]{(\!( #1 )\!)}
\newcommand{\rst}[2]{\left. #1 \right|_{#2}}

\newcommand{\Dqcl}[1]{\mathsf{D}^{\qc,l}_{#1}}
\newcommand{\Dqcr}[1]{\mathsf{D}^{\qc,r}_{#1}}

\DeclareMathOperator{\ad}{ad}
\DeclareMathOperator{\id}{id}

\DeclareMathOperator{\Sp}{Sp}
\DeclareMathOperator{\Aut}{Aut}
\DeclareMathOperator{\Der}{Der}

\DeclareMathOperator{\Hom}{Hom}

\DeclareMathOperator{\Ker}{Ker}
\DeclareMathOperator{\Lie}{Lie}

\DeclareMathOperator{\sord}{sord}

\DeclareMathOperator{\Spec}{Spec}

\DeclareMathOperator{\Wedge}{{\textstyle\bigwedge}}
\DeclareMathOperator{\CDR}{\EuScript{CDR}}

\DeclareMathOperator{\shBer}{\EuScript{B}\kern-.1em\mathit{er}}
\DeclareMathOperator{\shDer}{\EuScript{D}\kern-.1em\mathit{er}}
\DeclareMathOperator{\shEnd}{\EuScript{E}\kern-.1em\mathit{nd}}
\DeclareMathOperator{\shHom}{\EuScript{H}\kern-.15em\mathit{om}}
\DeclareMathOperator{\shSym}{\EuScript{S}\kern-.1em\mathit{ym}}

\begin{document}

\title{$N_K=1$ SUSY structure of chiral de Rham complex \\ from the factorization structure}
\author{Takumi Iwane, Shintarou Yanagida}
\date{2024.09.06}
\address{Graduate School of Mathematics, Nagoya University.
 Furocho, Chikusaku, Nagoya, Japan, 464-8602.}
\email{takumi.iwane.c8@math.nagoya-u.ac.jp, yanagida@math.nagoya-u.ac.jp}

\begin{abstract}
We elucidate the comment in (Kapranov-Vasserot, Adv.\ Math., 2011, Remark 5.3.4) 
that the $1|1$-dimensional factorization structure of the formal superloop space of 
a smooth algebraic variety $X$ induces the $N_K=1$ SUSY vertex algebra structure 
of the chiral de Rham complex of $X$.
\end{abstract}

\maketitle
{\small \tableofcontents}

\section{Introduction}\label{s:0}

This note is intended to elucidate the comment in \cite[Remark 5.3.4]{KV11}
concerning the chiral de Rham complex $\Omega_X^{\ch}$ of a smooth algebraic variety $X$
and the factorization structure of the global formal loop space $\LC{}X$ 
over a $1|1$-dimensional smooth supercurve $C$. 
Let us start the explanation with the background.
We work over the field $\bbC$ of complex numbers throughout this note.

\subsection*{Background}

In \cite{BD}, Beilinson and Drinfeld constructed a geometric theory of vertex algebras
by introducing (Beilinson-Drinfeld) chiral algebras and factorization algebras.
The theory is build on the language of $\shD$-modules, and gives a coordinate-free formulation.
See also \cite[Chap.\ 19,20]{FB} and \cite{FG} for the details. 

In \cite{MSV}, Malikov, Schechtman and Vaintrob introduced the chiral de Rham complex $\Omega_X^{\ch}$
of a smooth algebraic variety $X$. 
It is a sheaf of dg (differential graded) vertex algebras on $X$, 
and can be naively regarded as a loop space analogue, or a semi-infinite analogue, 
of the de Rham complex $\Omega_X$.

This naive analogy is clarified by Kapranov and Vasserot in \cite{KV04}, 
where the notion of a factorization space is introduced 
as a non-linear analogue of Beilinson-Drinfeld factorization algebras. 
In \cite{KV04}, the formal loop space $\clL X$ of a smooth algebraic variety $X$ and 
the global formal loop space $\clL_{C_0}X$ over a smooth curve $C_0$ are constructed as 
(tamely-behaving) ind-schemes, and a structure of a factorization space on $\clL_{C_0}X$ is introduced.
Moreover, the de Rham complex $\CDR_X$ (associated to the right $\shD_X$-module $\Omega_X^{\dim X}$) 
is constructed from the ind-scheme $\clL_{C_0}X$ as a sheaf of complexes on $X \times C_0$.
The fiber of $\CDR_X$ over a point of $C_0$, which is a sheaf of dg vertex algebras on $X$,
is identified with the chiral de Rham complex $\Omega_X^{\ch}$.
Finally, the factorization structure on $\clL_{C_0}X$ induces a chiral algebra 
structure on $\CDR_X$ over $C_0$,
which coincides with the vertex algebra structure on $\Omega_X^{\ch}$.
Thus, we have gained some insight into the geometric origin of 
the vertex algebra structure of $\Omega_X^{\ch}$.

The chiral de Rham complex $\Omega_X^{\ch}$ enjoys many interesting properties,
as studied extensively in these 25 years.
One of them is the supersymmetry.
To explain that, we recall the notion of a SUSY vertex algebra introduced by Heluani and Kac \cite{HK},
which is a superfield formulation of a vertex superalgebra equipped with supersymmetry.
There are two types of SUSY vertex algebras: $N_W=N$ and $N_K=N$ algebras, and what is called 
an $N$-supersymmetry in physics (and mathematics) literatures corresponds to the $N_K=N$ one.
Let us also mention the work of Heluani \cite[\S4]{H} 
which constructed the theory of chiral algebras on supercurves 
and investigated the relation to SUSY vertex algebras.

It is known that the chiral de Rham complex $\Omega_X^{\ch}$ has a natural structure 
of an $N_K=1$ SUSY vertex algebra \cite[Example 5.13]{HK}, \cite[\S4]{BHS}.
At this point, one may ask if this SUSY structure on $\Omega_X^{\ch}$ can be understood 
in terms of some geometric argument, which would be an extension of the argument in \cite{KV04}.

In \cite{KV11}, Kapranov and Vasserot generalized the formal loop space construction in \cite{KV04}
to the super setting.
As we will review in \cref{s:L}, they constructed the formal superloop space $\LS Y$ 
for a smooth superscheme $Y$ of finite type, 
and the global formal superloop space $\LC{}Y$ for $Y$ and a smooth supercurve $C$.
As will be reviewed in \cref{fct:FS:LSC}, 
they also introduced a structure of a factorization space on $\LC{}Y$.

As will be explained in \cref{ss:L:NW}, these results in \cite{KV11} and some materials 
from \cite{BD,KV04} imply that the chiral de Rham complex $\Omega_X^{\ch}$ 
of a smooth algebraic variety $X$ can be identified with the fiber of the sheaf $\CDR_X$
over a point of a smooth supercurve $C$ of dimension $1|1$,
and also with the fiber of the structure sheaf of the global formal superloop space $\LC{}X$. 
The factorization structure on $\LC{}X$ induces on the last one 
a natural structure $\mu$ \eqref{eq:FS:mu} of a chiral algebra over $C$.
According to \cite[\S4]{H}, such $\mu$ corresponds to an $N_W=1$ SUSY vertex algebra structure.
So, there is a slight difference on SUSY structures:
the natural one is $N_K=1$, but what is obtained is $N_W=1$.

Now, in \cite[Remark 5.3.4]{KV11}, 
it is commented that the factorization space structure on $\LC{}X$
gives ``a geometric reason'' for the natural $N_K=1$ SUSY structure on $\Omega_X^{\ch}$.
However, as we mentioned above, a straightforward argument only yields an $N_W=1$ structure.
As mentioned in the beginning, our goal is to elucidate this problem.

\subsection*{Main argument}

Our solution of this problem is given in \cref{s:NK}. The idea is simple: 
As the base supercurve $C$ of $\LC{}X$, we take a superconformal curve $(C,\shS)$.
Then, the $N_W=1$ SUSY chiral product $\mu$ \eqref{eq:FS:mu} over $C$ 
constructed from the factorization structure can be naturally modified to 
a superconformal one $\mu^{\tsc}$ \eqref{eq:NK:mu-sc} over $(C,\shS)$.
The modification is made by simply replacing the diagonal $\Delta \subset C^2$ 
with the superdiagonal $\Delta^{\tsc} \subset C^2$ (see \cref{dfn:SC:Dsc}).
Note that we used the superconformal structure $\shS$ to get the well-defined $\Delta^{\tsc}$.

\subsection*{Organization}

\cref{s:SC}--\cref{s:FS} are more or less the review of necessary materials,
and \cref{s:NK} gives the justification of the comment in \cite[Remark 5.3.4]{KV11}.

The starting \cref{s:SC} gives preliminaries from super algebraic geometry,
In \cref{ss:SC:pre}, we introduce basic notations.
In \cref{ss:SC:SC}, we introduce the notions of a smooth supercurve $C$ 
and of a superconformal curve $(C,\shS)$.
In \cref{ss:SC:D}, we briefly recall the theory of algebraic $\shD$-modules in the super setting.
\cref{ss:SC:SCD} gives a self-contained explanation of $\shD$-modules on a superconformal curve.

\cref{s:L} is a recollection of the formal superloop space $\LS X$ and the global one $\LC{}X$
from \cite{KV04,KV11}.

\cref{s:FS} recalls the notion of a factorization space and a factorization $\shD$-module 
from \cite{KV04,KV11}. 
In \cref{ss:L:NW}, we give a few comments on the explained materials.
In particular, we explain that the factorization space structure on $\LC{}X$ yields 
an $N_W=1$ SUSY vertex algebra structure on the chiral de Rham complex $\Omega_X^{\ch}$.

The final \cref{s:NK} gives our justification of the comment in \cite[Remark 5.3.4]{KV11}.

\subsection*{Global notation}

The following list gives an overview of the terminology and notations used in the text.
\begin{itemize}
\item
The symbol $\bbN$ denotes the set $\{0,1,2,\dotsc\}$ of all nonnegative integers.


\item 
A ring or an algebra means a unital associative one unless otherwise stated.

\item
A section of a sheaf $\shE$ means a local section unless otherwise stated,
and we denote $e \in \shE$ to indicate that $e$ is a section of $\shE$, i.e.,
there is an open subset $U$ such that $e\in\shE(U)$.

\item
For a category $\catC$, we denote $X \in \catC$ to say that $X$ is an object of $\catC$.
The morphism set in a category $\catC$ is denotes by $\Hom_{\catC}(\cdot,\cdot)$.

\item
We follow \cite{DM} for the terminology of super mathematics.
We denote the cyclic group of order $2$ by $\Zt=\{\ev,\od\}$,
and a $\Zt$-grading is called \emph{parity}.
The parity of an object $x$ is denoted by $\abs{x} \in \Zt$.

\end{itemize}

\section{Supercurves and superconformal curves}\label{s:SC}

Let us start with notations of super algebraic geometry.
We use the language given in \cite[\S1, \S3]{KV11}.
See also \cite{BH,BHP} for the foundation.
We refer to \cite[Chap.\ 2]{M} for the geometry of supercurves. 

\subsection{Preliminary of notations}\label{ss:SC:pre}

Following \cite[\S1.1]{KV11}, we denote a superscheme as $X=(\ul{X},\shO_X)$,
where $\ul{X}$ is the underlying topological space and $\shO_X$ is the structure sheaf.
The $\Zt$-structure of $\shO_X$ is denoted as $\shO_X = \shO_{X,\ev} \oplus \shO_{X,\od}$
with $\Zt=\{\ev,\od\}$. 
Given a superscheme $X$, its \emph{even part} is defined to be the scheme 
$X_{\even} \ceq (\ul{X},\shO_X/\shJ_X) = (\ul{X},\shO_{X,\ev}/\shO_{X,\od}^2)$, 
where $\shJ_X \subset \shO_X$ is the ideal generated by odd sections, 
and its \emph{reduced scheme} $X_{\tred}$ is defined to be 
\begin{align}\label{eq:SC:red}
 X_{\tred} \ceq (\ul{X},\shO_{X,\ev}/\sqrt{\shO_{X,\ev}}).
\end{align}

As given in \cite[\S1.2, p.1085]{KV11}, we have the notion of 
a \emph{smooth} (\emph{algebraic}) \emph{supervariety} $X$ over the complex number field $\bbC$.
The structure morphism $X \to \Spec\bbC$ is then a smooth morphism, and the dimension  $\dim X$
(relative over $\Spec\bbC$) is in $\bbN^2$, denoted as $p|q$.
For example, the \emph{affine superspace} of dimension $p|q$
\begin{align}\label{eq:SC:Apq}
 \bbA^{p|q}_\bbC \ceq \Spec(\bbC[x_1,\dotsc,x_p] \otimes_{\bbC} \bbC[\xi_1,\dotsc,\xi_q])
\end{align} 
is a smooth supervariety, 
where $x_i$'s are even and $\xi_j$'s are odd coordinate functions.

Hereafter we suppress the word ``over $\bbC$'' if confusion may not occur.
As explained in \cite[\S1.1, (1.1.2)]{KV11}, 
for a morphism $A \to B$ of supercommutative algebras, 
we have the \emph{module of K\"ahler differentials}
\[
 \Omega_{B/A}^1 \ceq I/I^2, \quad I \ceq \Ker(B \otimes_A B \xr{m} B),
\]
where $m$ denotes the multiplication map.
It is a $B$-module equipped with the universal differential $d\colon B \to \Omega_{B/A}^1$.
We understand that $d$ is \emph{even} (see \cite[\S1.2, p.103, Notation]{Y} for the detail).
As a global counterpart, for a morphism $X \to Y$ of superschemes, 
we have the \emph{sheaf $\Omega_{X/Y}^1$ of K\"ahler differentials} on $X$. 
It is a quasi-coherent $\shO_X$-module equipped 
with an even universal differential $d\colon \shO_X \to \Omega_{X/Y}^1$.
We also have the de Rham (super)complex 
\begin{align}\label{eq:SC:deRham}
 0 \lto \shO_X \xrr{d} \Omega_{X/Y}^1 \xrr{d} \Omega_{X/Y}^2 \xrr{d} \cdots, \quad 
 \Omega_{X/Y}^p \ceq \Wedge_{\shO_X}^p \Omega_{X/Y}^1.
\end{align}
Note that this complex is unbounded in general even if $X$ is smooth over $Y$.
If $Y=\Spec\bbC$, we denote
\begin{align}\label{eq:SC:Omega}
 \Omega_X^1 \ceq \Omega_{X/\Spec\bbC}^1, \quad 
 \Omega_X^p \ceq \Wedge_{\shO_X}^p \Omega_X^1 \quad (p \in \bbN).
\end{align}

Let $X$ be a smooth supervariety of dimension $p|q$.
By \cite[\S1.2]{KV11}, the sheaf $\Omega_X^1$ is locally free of rank $p|q$.
The $\shO_X$-dual of $\Omega_X^1$ is the the \emph{tangent sheaf} $\Theta_X$ of $X$, 
and it is isomorphic to the sheaf $\shDer_{\bbC}(\shO_X)$ of derivations of $\shO_X$ over $\bbC$:
\[
 \Theta_X \ceq \shHom_{\shO_X}(\Omega_X^1,\shO_X) \cong \shDer_{\bbC}(\shO_X).
\] 
The tangent sheaf is a quasi-coherent $\shO_X$-module 
(regarded as $\shHom_{\shO_X}(\Omega_X^1,\shO_X)$)
and a sheaf of Lie $\bbC$-superalgebras on $X$ (regarded as $\shDer_{\bbC}(\shO_X)$).
It is moreover a Lie superalgebroids on $X$ (see \cite[2.9.1, 2.9.2]{BD} for Lie algebroids).

Let $X$ be as before.
By \cite[Proposition A.17]{BHP} (see also \cite[Tag 054L]{S} for the non-super setting), 
for every point $x$ of $X$, there exists an \'etale morphism 
\begin{align}\label{eq:SC:Z}
 Z=(z_1,\dotsc,z_p,\zeta_1,\dotsc,\zeta_q)\colon U \lto \bbA_{\bbC}^{p|q} 
\end{align}
from some affine open $U \subset X$ containing $x$,  
with $z_i$'s even and $\zeta_j$'s odd, such that they generate the maximal ideal 
$\frm_x \subset \shO_{X,x}$  at $x$ and the differentials 
\[
 dZ = (dz_1,\dotsc,dz_p,d\zeta_1,\dotsc,d\zeta_q)
\]
form a basis of the linear superspace $\Omega_{X,x}^1$ of K\"ahler differentials.
The corresponding derivations 
\begin{align}\label{eq:SC:pdZ}
 \pd_Z = (\pd_{z_1},\dotsc,\pd_{z_p},\pd_{\zeta_1},\dotsc,\pd_{\zeta_q}),
\end{align}
where $\pd_{z_1} \ceq \frac{\pd}{\pd z_1}$ and so on, 
form a basis of the tangent superspace $\Theta_{X,x}$.
We call the morphism $Z$ (or the pair $(U,Z)$) a \emph{local coordinate system} at $x$.

\subsection{The differential operators \texorpdfstring{$\shD$}{D} in the super setting}\label{ss:SC:D}

Let $X$ be a smooth supervariety over $\bbC$, 
and $\shEnd_{\bbC}(\shO_X)$ be the endomorphism sheaf of the superalgebra $\shO_X$.
The \emph{sheaf of differential operators on $X$} is the subsheaf 
\begin{align}\label{eq:SC:D}
 \shD_X \subset \shEnd_{\bbC}(\shO_X)
\end{align}
generated by $\shO_X$ (considered as multiplication operators) 
and $\Theta_X \cong \shDer_{\bbC}(\shO_X)$ (considered as derivations).

In this note, a \emph{quasi-coherent left} (resp.\ \emph{right}) $\shD_X$-module means 
a left (resp.\ right) $\shD_X$-module which is quasi-coherent as an $\shO_X$-module.
We denote the category of quasi-coherent left, resp.\ right, $\shD_X$-modules by 
\begin{align*}
 \Dqcl{X}, \quad \Dqcr{X}.
\end{align*}
In particular, $\Dqcl{X}$ is the category of quasi-coherent $\shO_X$-modules
equipped with (left) integrable connections along $X$.

As in the non-super case (\cite[\S1.2]{HTT}, \cite[2.1.1]{BD}), if $\shL,\shL' \in \Dqcl{X}$, 
then $\shL \otimes \shL' \ceq \shL \otimes_{\shO_X} \shL'$ is a left $\shD_X$-module.
Thus $\Dqcl{X}$ is a monoidal category with unit object $\shO_X$.
Also, if $\shL \in \Dqcl{X}$ and $\shM \in \Dqcr{X}$, 
then $\shM \otimes_{\shO_X} \shL$ is a right $\shD_X$-module on which  
$\tau \in \Theta_X \subset \shD_X$ acts as 
$(m \otimes l)\tau = (m\tau) \otimes l - m \otimes(\tau l$).
We denote it by $\shM \otimes \shL \in \Dqcr{X}$.
Moreover, if $\shM,\shM'\in\Dqcl{X}$, then $\shHom_{\shO_X}(\shM,\shM') \in \Dqcl{X}$ by
$(\tau f)(m)=-f(m)\tau+f(m\tau)$.

We denote the \emph{Berezinian sheaf of $X$} by 
\begin{align}\label{eq:SC:Ber}
 \omega_X \ceq \shBer(\Omega_X^1).
\end{align}
It is a right $\shD_X$-module by $\nu\tau \ceq -\Lie_\tau(\nu)$ for $\nu \in \omega_X$ and
$\tau \in \Theta_X$, where $\Lie$ denotes the Lie derivative.
See \cite[\S1.10, \S1.11]{DM} and \cite{N} for the details of $\omega_X$.
As shown in \cite{P}, the categories $\Dqcl{X}$ and $\Dqcr{X}$ are equivalent by 
\begin{align}\label{eq:D:L=R}
\begin{split}
 \Dqcl{X} \lsto \Dqcr{X}, \quad 
&\shL \lmto \shL^r \ceq \omega_{X} \otimes \shL, \\
 \Dqcr{X} \lsto \Dqcl{X}, \quad 
&\shM \lmto 
 \shM^l \ceq \omega_X^\vee \otimes_{\shO_X} \shM \cong \shHom_{\shO_X}(\omega_{X},\shM),
\end{split}
\end{align}
If $q=0$, i.e., $X$ is even, then $\omega_X$ is equal to the top differential forms $\Omega_X^p$.
(c.f.\ \eqref{eq:SC:deRham}).
The category $\Dqcr{X}$ is a monoidal category with 
$\shM \otimes \shM' \ceq \bigl((\shM^l) \otimes (\shM'^l)\bigr)^r \in \Dqcr{X}$
and unit object $\omega_X$.

For a morphism $f\colon X \to Y$ of smooth superschemes, we denote the (underived) direct and 
inverse image functors for (left or right) $\shD$-modules by $f_*$ and $f^*$, respectively.
For a left $\shD_Y$-module $\shL$, the $\shD$-module inverse image $f^*\shL$ is given 
by the same formula as an $\shO_Y$-module, equipped with a natural left $\shD_X$-module structure.
The right $\shD$-module inverse image is then given by the left-right equivalence \eqref{eq:D:L=R}.
For a right $\shD_X$-module $\shM$,  the $\shD$-module direct image $f_*\shM$ is the $\shO_Y$-module 
\begin{align}\label{eq:SC:f_*}
 f_\cdot(\shL \otimes_{\shD_X}\shD_{X \to Y}), \quad 
 \shD_{X \to Y} \ceq \shO_X \otimes_{f^{-1}\shO_Y}f^{-1}\shD_Y, 
\end{align}
equipped with a natural right $\shD_Y$-module structure.
Here $f_\cdot$ denotes the $\shO$-module direct image.
The left $\shD$-module direct image is then given by the left-right equivalence \eqref{eq:D:L=R}.

See \cite{P} for the details of the $\shD$-module direct and inverse image functors
(although \cite{P} develops the theory in the super-complex-analytic setting, 
the translation to the super-algebraic setting is straightforward).
The basic properties of the non-super versions of these functors 
are explained in \cite[Chap.\ 1]{HTT} and \cite[2.1.2--2.1.4]{BD},
and they hold similarly in the super case.

\subsection{Supercurves and superconformal curves}\label{ss:SC:SC}

\begin{dfn}
Let $N \in \bbZ_{>0}$.
A \emph{smooth supercurve of dimension $1|N$} is a smooth supervariety $C$ with $\dim(C)=1|N$.
A \emph{superconformal curve of dimension $1|N$} is 
a smooth supercurve $C$ of dimension $1|N$ equipped with a locally free subsheaf 
$\shS \subset \Theta_C$ of rank $0|N$ such that the composition map 
$\shS \otimes_{\bbC} \shS \xr{[\cdot,\cdot]} \Theta_C \xr{\bmod \shS} \Theta_C/\shS$
induces an $\shO_C$-module isomorphism $\shS \otimes_{\bbC} \shS \sto \Theta_C/\shS$.
\end{dfn}

As explained in \cite[\S4]{H}, we can consider an $N_W=N$ SUSY vertex algebras 
as a chiral algebra on a smooth supercurve of dimension $1|N$,
and an $N_K=N$ SUSY vertex algebra as one on a superconformal curve of dimension $1|N$.

In this note, we only consider superconformal curves of dimension $1|1$,
and suppress the term ``of dimension $1|1$'' if no confusion may occur.
They are also called SUSY curves (in \cite{BH,BHP} for example),
and called $\text{SUSY}_1$ curves in \cite{M}.
In the complex analytic setting, they are called super Riemann surfaces (in \cite{W} for example).

Let us recall local properties of superconformal curves.

\begin{lem}\label{lem:SC:SDer}
Let $(C,\shS)$ be a superconformal curve.
\begin{enumerate}
\item \label{i:SDer:1}
For each point $x$ of $C$, there is a local coordinate system $Z=(z,\zeta)$ 
at $x$ such that $\shS$ is locally generated as an $\shO_C$-module by the odd derivation
\begin{align}\label{eq:SC:Zsc}
 \sd_Z  \ceq \pd_{\zeta}  +\zeta  \pd_z.
\end{align}

\item \label{i:SDer:20}
The $\shO_C$-dual $\shS^\vee \ceq \shHom_{\shO_C}(\shS,\shO_C)$ 
satisfies $\shS^\vee \cong \omega_C$.

\item \label{i:SDer:2}
Let $\tau\colon \Omega_C^1 \to \shS^\vee$ be the dual of the injection $\shS \inj \Theta_C$.
Then the composition 
\[
 \Ker\tau \linj \Omega_C^1 \xrr{d} \Omega_C^2 \xrr{\tau \wedge \tau} (\shS^\vee)^{\otimes 2}
\]
yields an isomorphism $\Ker\tau \sto (\shS^\vee)^{\otimes 2}$.
Also, under the pairing $\Theta_C$ and $\Omega_C^1$, the sheaves $\shS$ and $\Ker\tau$ are normal.
Moreover, in the local coordinate system $Z$ in \ref{i:SDer:1},
$\Ker\tau$ is generated by
\begin{align}\label{eq:SC:varpi}
 \vp \ceq dz-\zeta d\zeta,
\end{align}
and this form is globally defined up to multiplication by an even function on $C$.

\item \label{i:SDer:3}
The Lie superalgebra $\shDer_{\bbC}(\shO_C) \cong \Theta_C$ on $C$ 
is generated by the subsheaf $\shS \subset \Theta_C$.

\item \label{i:SDer:4}
The sheaf $\shD_C$ of differential operators on $C$ is generated as a sub-superalgebra of 
$\shEnd_{\bbC}(\shO_C)$ by $\shO_C$ (regarded as multiplication operators) and $\shS$.
\end{enumerate}
\end{lem}

\begin{proof}
\begin{enumerate}
\item 
It is well-known. See \cite[\S2.1]{W} for example. 

\item 
See \cite[Chap.\ 2, \S6.1]{M} and \cite[Proposition 3.5]{BH}.

\item 
See \cite[Proposition 3.5]{BH} and \cite[\S2.3]{W}.

\item
Take an affine open covering $\{U_i\}_{i \in I}$ of $C$ which consist of local coordinate systems 
satisfying the condition in \ref{i:SDer:1}.
Then the statement from the relations
$\pd_z=\frac{1}{2}[\sd_Z,\sd_Z]$ and $\pd_{\zeta}=\sd_Z-\zeta\pd_z$ in 
the Lie superalgebra $\shEnd_{\bbC}(\shO_C)(U_i)$ for each $i \in I$.

\item
This follows from \ref{i:SDer:3}.
\end{enumerate}
\end{proof}

\begin{dfn}\label{dfn:SC:sc}
We call the local coordinate system $Z$ in \cref{lem:SC:SDer} \ref{i:SDer:1}
a \emph{superconformal coordinate system} on a superconformal curve $(C,\shS)$.
The form $\vp$ in \eqref{eq:SC:varpi} which generates $\Ker \tau \subset \Omega$ is called 
the \emph{contact form} of $(C,\shS)$. 
\end{dfn}

\begin{rmk}
We follow \cite{V} to use the terminology ``contact form'' for $\vp$.
The statement of \cref{lem:SC:SDer} \ref{i:SDer:4} is mentioned in \cite[Chap.\ 2, \S4.2]{M}.
\end{rmk}

Given a superconformal coordinate system $Z$, by \cref{lem:SC:SDer} \ref{i:SDer:4}, 
any section $K$ of $\shD_C$ can be locally written in the form 
$K=\sum_{j=0}^d a_j \sd_Z^j$ with $a_d \ne 0$.
Then we define $\sord_Z(p) \ceq d/2$.
For example, we have $\sord_Z(\pd_z)=\sord_Z(\pd_\zeta)=1$. 
The definition of $\sord_Z$ seemingly depends on the choice of $Z$, but we have:

\begin{fct}[{\cite[Chap.\ 2, 4.3]{M}}]\label{fct:SC:sord}
Let $Z=(z,\zeta)$ and $Z'=(z',\zeta')$ be two superconformal coordinate systems 
at a point of a superconformal curve $(C,\shS)$.
Then, the two filtrations of $\shD_C$ by $\sord_Z$ and $\sord_{Z'}$ coincide.
%
%
\end{fct}


Next, we introduce the superdiagonal of a superconformal curve $(C,\shS)$.
We denote by
\begin{align}\label{eq:SC:D-U}
 \Delta\colon C \linj C^2, \quad j\colon U \linj C^2
\end{align}
the diagonal embedding and the open embedding of the complement.
Abusing notation, we denote the image of $\Delta$ by the same symbol as 
$\Delta \ceq \Delta(C) \subset C^2$.
We also denote the defining ideal of $\Delta \subset C^2$ by 
\[
 \shI_\Delta \subset \shO_{C^2}.
\]
Similarly as in the non-super case, we have 
$\Delta_*(\Omega_C^1) \cong \shI_\Delta/\shI_\Delta^2$ as $\shO_{C^2}$-modules.

We denote the first-order infinitesimal neighborhood of $\Delta$ by $\Delta^{(1)}$.
In other words, $\Delta^{(1)}$ is the closed sub-superscheme of $C^2$ 
defined by the ideal $\shI_\Delta^2$.
We have an exact sequence of $\shO_{C^2}$-modules
\[
 0 \lto \shI_\Delta/\shI_\Delta^2 \lto \shO_{\Delta^{(1)}} \lto \shO_\Delta \lto 0.
\]
Now, consider the composition $\Omega_C^1 \to \shS^\vee \cong \omega_C$ of 
the morphism $\Omega_C^1 \to \shS^\vee$ induced by the inclusion $\shS \inj \Theta_C$ 
and the isomorphism $\shS^\vee \cong\omega_C$ in \cref{lem:SC:SDer} \ref{i:SDer:20}.
We denote the kernel of the composition by 
\begin{align}\label{eq:SC:shK}
 \shK \ceq \Ker(\Omega^1_C \to \omega_C).
\end{align}
Then, we define a closed sub-superscheme $\Delta^{\tsc} \subset C^2$ by 
$\shO_{\Delta^{\tsc}} = \shO_{\Delta^{(1)}}/\shK$,
and denote the defining ideal by $\shI_{\Delta^{\tsc}} \subset \shO_{C^2}$.
Hence, we have
\[
 \shO_{\Delta^{\tsc}} = \shO_{\Delta^{(1)}}/\shK = \shO_{C^2}/\shI_{\Delta^{\tsc}}.
\]

\begin{dfn}\label{dfn:SC:Dsc}
We call the closed sub-superscheme $\Delta^{\tsc}$ 
the \emph{superdiagonal} of the superconformal curve $(C,\shS)$.
Abusing notation, we denote the closed embedding by $\Delta^{\tsc}\colon \Delta^{\tsc} \inj C^2$.
\end{dfn}

\begin{fct}[{\cite[Chap.\ 2, 6.3]{M}}]\label{fct:SC:Z}
Let $(C,\shS)$ and $\Delta^{\tsc}$ be as above.
Let $Z=(z,\zeta)$ be a superconformal coordinate of $(C,\shS)$, 
and put $Z_i = (z_i,\zeta_i) \ceq p_i^*(Z)$ for $i=1,2$, 
where $p_i\colon C^2 \to C$ are the projections.
Then $\Delta^{\tsc} \subset C^2$ is locally defined by the equation
\[
 z_1-z_2-\zeta_1\zeta_2=0.
\]
\end{fct}

Thus, $\Delta^{\tsc} \subset C^2$ is a closed subscheme of codimension $1|0$, and the underlying 
topological space of $\Delta^{\tsc}$ is the set-theoretic diagonal $\ul{\Delta} \subset \ul{C}^2$.
Also, for the defining ideals $\shI_{\Delta},\shI_{\Delta^{\tsc}} \subset \shO_{C^2}$, we have
\begin{align}\label{eq:SC:ID-IDs}
 \shI_{\Delta}^2 \subset \shI_{\Delta^{\tsc}} \subset \shI_{\Delta}.
\end{align}
In fact, $\shI_\Delta$ is locally generated by $z_1-z_2$ and $\zeta_1-\zeta_2$,
and the relation $\shI_{\Delta^{\tsc}} \subset \shI_{\Delta}$ 
follows from $\zeta_1\zeta_2=(\zeta_1-\zeta_2)\zeta_2$.
The other relation follows from $(z_1-z_2)^2=(z_1-z_2-\zeta_1\zeta_2)(z_1-z_2+\zeta_1\zeta_2)$,

Recalling the morphism $\Omega_C^1 \to \omega_C$ considered in \eqref{eq:SC:shK}, 
we denote the composition of it with the de Rham differential $d\colon \shO_C \to \Omega_C^1$ by
\begin{align}\label{eq:SC:delta_Z}
 \delta\colon \shO_C \lto \omega_C.
\end{align}
In a superconformal coordinate system $Z=(z,\zeta)$, it is given by 
\[
 f \lmto \delta f = dZ \cdot \sd_Z f, \quad \sd_Z = \pd_\zeta+\zeta\pd_z,
\]
where $dZ$ is the local generator of the $\shO_C$-module $\omega_C$.
Hence, $\delta$ is an odd derivation.

%
%

\subsection{\texorpdfstring{$\shD$}{D}-modules on superconformal curves}\label{ss:SC:SCD}

Let $(C,\shS)$ be a superconformal curve, and $\Delta^{\tsc}\subset C^2$ the superdiagonal.
We denote by $U^{\tsc} \subset C^2$ the complement of $\Delta^{\tsc}$,
which is locally given by $z_1-z_2-\zeta_1\zeta_2 \ne 0$ in a superconformal coordinate
$Z_i=(z_i,\zeta_i) \ceq p_i^*(Z)$, $Z=(z,\zeta)$.
Abusing the notation, we denote the closed embedding of $\Delta^{\tsc}$ 
and the open embedding of $U^{\tsc}$ by 
\begin{align}\label{eq:SC:D-U-sc}
 \Delta^{\tsc}\colon \Delta^{\tsc} \linj C^2, \quad 
 j^{\tsc}\colon U^{\tsc} \linj C^2
\end{align}
Now, consider the direct image $\Delta^{\tsc}_*\shM$ of a right $\shD_C$-module $\shM$.
From the definition \eqref{eq:SC:f_*}, 
in a superconformal coordinate $Z=(z,\zeta)$, it is locally given by
$\Delta_{\cdot}\shM \otimes_{\bbC[\sd_Z]} \bbC[\sd_{Z_1},\sd_{Z_2}]$, 
where $\Delta_{\cdot}$ denotes the $\shO$-module direct image under the ordinary diagonal map $\Delta$.
The action of $\sd_Z$ on $\bbC[\sd_{Z_1},\sd_{Z_2}]$ is given by $\sd_{Z}=\sd_{Z_1}+\sd_{Z_2}$.
Similarly, the direct image $\Delta^s_*\shL$ of a left $\shD_C$-module $\shL$
is locally given by $\bbC[\sd_{Z_1},\sd_{Z_2}] \otimes_{\bbC[\sd_Z]} \Delta_{\cdot}\shL$.

We can rewrite these direct images using the formal delta function \cite[(1.6.1)]{HK}.
Let $Z_1$ and $Z_2$ be as above, and consider the formal distribution 
\[
 \delta(Z_1,Z_2) \ceq (i_{z_1,z_2}-i_{z_2,z_1})\frac{\zeta_1-\zeta_2}{z_1-z_2},
\]
where $i_{z_1,z_2}$ denotes the expansion in the domain $\abs{z_1}>\abs{z_2}$,
and similar for $i_{z_2,z_1}$.
By \cite[3.1.4, Lemma 4.2]{HK}, we have 
\[
 (z_1-z_2-\zeta_1\zeta_2)\delta(Z_1,Z_2) = (\sd_{Z_1}+\sd_{Z_2})\delta(Z_1,Z_2) = 0.
\]
Thus, the direct image of the left $\shD$-module $\shO_C$ under $\Delta^{\tsc}$ 
can be identified with the left $\shD$-module generated by $\delta(Z_1,Z_2)$:
\begin{align}\label{eq:SC:DO=Dd}
 \Delta^{\tsc}_*\shO_C \cong \shD_{C^2} \cdot \delta(Z_1,Z_2).
\end{align}
Similarly, we have an isomorphism of right $\shD$-modules
\begin{align}\label{eq:SC:Do=dD}
 \Delta^{\tsc}_*\omega_C \cong \delta(Z_1,Z_2)dZ_1dZ_2 \cdot \shD_{C^2}.
\end{align}

Following \cite[(4.1.1), (4.1.5)]{HK}, for $j \in \bbZ$ and $J \in \{0,1\}$, we denote
\[
 (Z_1-Z_2)^{j|J} \ceq (z_1-z_2-\zeta_1\zeta_2)^j(\zeta_1-\zeta_2)^J, \quad 
 \sd_Z^{(j|J)} \ceq (-1)^{J(J+1)/2}\tfrac{1}{j!} \pd_z^j \sd_Z^J,
\]
where we understand 
$(Z_1-Z_2)^{-1|0} = (z_1-z_2-\zeta_1\zeta_2)^{-1} = (z_1-z_1)^{-1}(1+\zeta_1\zeta_2)$,
and similarly for general $j<0$.
Then, similarly as in the non-super case \cite[19.1.5]{FB}, the correspondences
\begin{align*}
 (Z_1-Z_2)^{-j-1|J}         &\lmto \sd_{Z_2}^{(j|J)}\delta(Z_1,Z_2), \\
 (Z_1-Z_2)^{-j-1|J}dZ_1dZ_2 &\lmto \sd_{Z_2}^{(j|J)}\delta(Z_1,Z_2) \cdot dZ_1dZ_2
\end{align*}
for $j \in \bbN$ and $J \in \{0,1\}$ induce an isomorphism of left (resp.\ right) $\shD$-module
\begin{align*}
 (\shO_C \boxtimes   \shO_C)(\infty\Delta^{\tsc})/(\shO_C \boxtimes \shO_C) 
 &\lsto \shD_{C^2}\cdot\delta(Z_1,Z_2), \\
 (\omega_C \boxtimes \omega_C)(\infty\Delta^{\tsc})/(\omega_C \boxtimes \omega_C)  
 &\lsto \delta(Z_1,Z_2)dZ_1dZ_2\cdot\shD_{C^2}.
\end{align*}
Here $\boxtimes$ denotes the exterior product with respect to the projections $p_1,p_2\colon C^2 \to C$.
Also, for an $\shO_{C^2}$-module $\shF$, we denote the sheaf of meromorphic sections of $\shF$
with poles along $\Delta^{\tsc}$ of arbitrary order by
\begin{align}\label{eq:inf-Delta-sc}
 \shF(\infty\Delta^{\tsc}) \ceq \lim_n \shHom_{\shO_{C^2}}(\shI_{\Delta^{\tsc}}^n,\shF).
\end{align}
Combining these isomorphisms with \eqref{eq:SC:DO=Dd} and \eqref{eq:SC:Do=dD}, we have 
\begin{align}\label{eq:SC:rDs_*om}
 \Delta^{\tsc}_*\shO_C \cong (\shO_C \boxtimes \shO_C)(\infty\Delta^{\tsc})
 /(\shO_C \boxtimes \shO_C), \quad
 \Delta^{\tsc}_*\omega_C \cong (\omega_C \boxtimes \omega_C)(\infty\Delta^{\tsc})
 /(\omega_C \boxtimes \omega_C).
\end{align}
These isomorphisms can be generalized in a straightforward manner as: 
 
\begin{lem}[{c.f.\ \cite[4.1.2]{H}}]
For a quasi-coherent left $\shD_C$-module $\shL$, we have
\begin{align*}
 \Delta^{\tsc}_*\shL \cong (\shO_C \boxtimes \shL)(\infty\Delta^{\tsc})/(\shO_C \boxtimes \shL).
\end{align*}
as left $\shD_{C^2}$-modules.
Also, for a quasi-coherent right $\shD_C$-module $\shM$, we have
\begin{align}\label{eq:SC:rDs_*}
 \Delta^{\tsc}_*\shM \cong (\omega_C \boxtimes \shM)(\infty\Delta^{\tsc})/(\omega_C \boxtimes \shM).
\end{align}
\end{lem}

Let $\mu_\omega$ be the composition of the projection and 
the second isomorphism in \eqref{eq:SC:rDs_*om}:
\begin{align}\label{eq:SC:mu-om-sc}
 \mu_\omega^{\tsc}\colon (\omega_C \boxtimes \omega_C)(\infty\Delta^{\tsc}) 
 \xr{\bmod \omega_C \boxtimes \omega_C}
 (\omega_C \boxtimes \omega_C)(\infty\Delta^{\tsc})/(\omega_C \boxtimes \omega_C) \lsto 
 \Delta^{\tsc}_*\omega_C.
\end{align}
By \cite[4.1.7, 4.1.10]{H}, this gives a \emph{chiral algebra} structure on $\omega_C$.
That is, $\mu_\omega^{\tsc}$ satisfies the skew-symmetry and the Jacobi identity (in the chiral setting).

For later reference, let us modify the construction of $\mu_\omega^{\tsc}$ by using 
the ordinary diagonal embedding $\Delta\colon C \inj C^2$ \eqref{eq:SC:D-U} instead of $\Delta^{\tsc}$.
Then, similarly as \eqref{eq:inf-Delta-sc}, denoting 
the sheaf of meromorphic sections with poles along $\Delta$ by 
\[
 \shF(\infty\Delta) \ceq \lim_n \shHom_{\shO_{C^2}}(\shI_{\Delta}^n,\shF),
\]
we obtain
\begin{align}\label{eq:SC:mu-omega}
 \mu_\omega\colon (\omega_C \boxtimes \omega_C)(\infty\Delta) 
 \xr{\bmod \omega_C \boxtimes \omega_C}
 (\omega_C \boxtimes \omega_C)(\infty\Delta)/(\omega_C \boxtimes \omega_C) \lsto 
 \Delta_*\omega_C.
\end{align}
It is also a chiral algebra structure on $\omega_C$, but 
this structure exists for any smooth supercurve $C$ of dimension $1|1$ (not a superconformal one).

Finally, we note the right $\shD$-module isomorphisms 
\begin{align}\label{eq:SC:jj}
 (\omega_C \boxtimes \shM)(\infty\Delta)/(\omega_C \boxtimes \shM) \cong j_*j^*\shM, \quad 
 (\omega_C \boxtimes \shM)(\infty\Delta^{\tsc})/(\omega_C \boxtimes \shM) \cong j^{\tsc}_*{j^{\tsc}}^*\shM
\end{align}
for a right $\shD_C$-module $\shM$, which hold since we are working in the algebraic setting 
(in the complex-analytic setting, these do not hold).
In the literature \cite{BD,KV04,H}, chiral products $\mu$ are defined using these expressions.

\section{Formal loop space}\label{s:L}

We give a brief review of the formal loop spaces in the super setting, 
following \cite[\S1, \S2]{KV04} and \cite[\S4, \S5]{KV11}.

We continue to work over $\bbC$. 
In particular, linear superspaces, superalgebras and superschemes 
are defined over $\bbC$ unless otherwise stated.
We denote by $\SSch=\SSch_{\bbC}$ the category of superschemes over $\bbC$, 
which is a full subcategory of the category $\SSp=\SSp_{\bbC}$ of superspaces
(see \cite[\S1.1]{KV11} for the detail).
We also use the language of \emph{ind-superschemes}. See \cite[\S3]{KV11} for the foundation.

\subsection{Formal jet space, formal loop space, and the de Rham spectrum}\label{ss:L:JLS}

For a supercommutative algebra $R$ and an even indeterminate $t$, we denote by $R\dbr{t}$ 
the commutative superalgebra of formal power series in $t$ with coefficients in $R$.
For a sheaf $\shO$ of supercommutative algebras over some topological space, 
we denote by $\shO\dbr{z}$ the sheaf given by the correspondence 
$U \mapsto \shO(U)\dbr{t}$ for each open set $U$.
If $\shO$ is a sheaf of local supercommutative algebras,
then so is $\shO\dbr{t}$ by \cite[Lemma 4.1.1]{KV11}.

For a superscheme $X$, we denote by $\clJ X$ the (formal) \emph{jet space} of $X$ 
\cite[Proposition 4.2.1]{KV11}.
It is a superscheme representing the functor
\begin{align*}
 \SSch \lto \cSet, \quad S=(\ul{S},\shO_S) \lmto \Hom_{\SSp}((\ul{S},\shO_S\dbr{t}),X).
\end{align*} 
In particular, we have a morphism of superschemes
\begin{align}\label{eq:L:pJ}
 p_{\clJ}\colon \clJ X \lto X
\end{align}
corresponding to the surjection $\shO_X\dbr{t} \to \shO_X \cong \shO_X\dbr{t}/(t)$.
By \cite[Proposition 1.2.1 (d)]{KV04}, the morphism $p$ is affine.
We call $p$ the projection.
The construction $X \mto \clJ X$ is functorial.
If $X$ is a scheme, then $\clJ X$ is the scheme of jets (or arcs) in $X$.
We refer to \cite[\S1.3.1]{Y} for an explicit description of $\clJ X$ 
in the case where $X$ is an affine superscheme.
See also \cite[2.3.2, 2.3.3]{BD} for the functor $\clJ$ (in the non-super setting).

Next, for a supercommutative algebra $R$, we denote by $R\dpr{t}$ 
the supercommutative algebra of formal Laurent series $r(t)=\sum_{i\gg-\infty}^\infty r_it^i$
with even generator $t$ and coefficients $r_i \in R$, 
and by $R\dpr{t}^{\sqr}$ the sub-superalgebra of $R\dpr{t}$ consisting of the series $r(t)$
such that $r_i$ is nilpotent for any $i<0$.
By \cite[Corollary 1.3.2]{KV04}, if $R$ is a local commutative superalgebra,
then so is $R\dpr{t}^{\sqr}$.
For a sheaf $\shO$ of commutative superalgebras, we denote by $\shO\dpr{t}^{\sqr}$
the sheaf given by the correspondence $U \mapsto \shO(U)\dpr{t}^{\sqr}$.
If $\shO$ is a sheaf of local commutative superalgebras, then so is $\shO\dpr{t}^{\sqr}$

For a superscheme $X$ of finite type (over $\bbC$) \cite[\S1.1, p.1083]{KV11}, 
we denote by $\clL X$ the \emph{formal loop space} of $X$ \cite[(4.2.2), Proposition 4.2.3]{KV11}.
It is an ind-superscheme representing the functor
\begin{align*}
 \SSch \lto \cSet, \quad S \lmto \Hom_{\SSp}((\ul{S},\shO_S\dpr{t}^{\sqr}),X).
\end{align*}
More precisely, $\clL X$ is an ind-object in the category $\SSch$ represented as 
a filtering inductive limit of nilpotent extensions of the jet space $\clJ X$.
Hence, we have a diagram of ind-superschemes 
\begin{equation*}
\begin{tikzcd}
 X & \clJ X \ar[l,"p_{\clJ}"'] \ar[r,"i"] & \clL X 
\end{tikzcd}
\end{equation*}
with $p_{\clJ}$ the projection \eqref{eq:L:pJ} and $i$ a closed embedding.
Also, the construction $X \mto \clL X$ is functorial.

Next, we turn to the de Rham spectrum.
For an odd indeterminate $\tau$, we denote by $\Lambda[\tau]=\Lambda_{\bbC}[\tau]$ 
the exterior algebra (over $\bbC$) with one generator $\tau$.
Thus, it is the supercommutative algebra and equal to $\bbC+\bbC \tau$ as a $\bbC$-module.
The corresponding spectrum is the $0|1$-dimensional affine superspace \eqref{eq:SC:Apq}: 
$\Spec \Lambda[\tau] = \bbA^{0|1}$. 
For a superscheme $S=(\ul{S},\shO_S)$, the fiber product 
$S \times \bbA^{0|1} = S \times_{\bbC} \bbA^{0|1}$ is isomorphic to 
$(\ul{S},\shO_S \otimes \Lambda[\tau])$ by \cite[Proposition 2.1.2 (a)]{KV11}.

For a superscheme $X$, we denote by $\clS X$ the \emph{de Rham spectrum} of $X$ \cite[\S2.2]{KV11}.
It is a superscheme representing the functor 
\begin{align*}
 \SSch \lto \cSet, \quad S \lmto  \Hom_{\SSch}(\bbA^{0|1} \times S,X).
\end{align*}
The de Rham spectrum $\clS X$ has an explicit description \cite[Proposition 2.2.3]{KV11}.
To explain it, let $\Omega_X^1$ be the sheaf of K\"ahler differentials on $X$ \eqref{eq:SC:Omega},
regarded as an object with even parity and cohomological degree $0$.
Then, the \emph{de Rham superalgebra} on $X$ is defined to be 
\begin{align}\label{eq:L:OmegaX}
 \Omega_X \ceq \shSym_{\shO_X}(\Pi\Omega_X^1).
\end{align}
Here $\shSym_{\shO_X}$ denotes the symmetric $\shO_X$-algebra (functor), 
and $\Pi$ denotes the parity change functor. 
Thus, as a sheaf of modules, it is the sheaf of differential forms on $X$.
The unique closed point of $\bbA^{0|1}$ (the origin) determines a morphism 
\begin{align}\label{eq:L:pS}
 p_{\clS}\colon \clS X \lto X
\end{align}
called the projection. 
Now, by \cite[Proposition 2.2.3]{KV11}, we have an isomorphism of $\shO_X$-modules 
\begin{align}\label{eq:L:vpO=Om}
 {p_{\clS}}_*\shO_{\clS X} \cong \Omega_X.
\end{align}
This is the description of $\clS X$ mentioned above.

Moreover, the identification \eqref{eq:L:vpO=Om} holds as a commutative dg $\shO_X$-algebras
\cite[Proposition 2.2.2]{KV11}.
Here the dg (differential graded) structure on $\Omega_X$ is given by the de Rham differential $d$
and the degree of differential forms (see \eqref{eq:SC:deRham}).
The dg structure on ${p_{\clS}}_*\shO_{\clS X}$ is induced by the action of 
the automorphism group superscheme $\Aut(\bbA^{0|1})$ 
(denoted by $\ul{\Aut}(\bbA^{0|1})$ in \cite{KV11}),
due to the observation of Kontsevich.


Then next lemma can be proved easily using the characterization of 
smoothness of a superscheme \cite[Proposition A.17]{BHP}.

\begin{lem}
Let $X$ be a (non-super) smooth scheme of finite type.
Then, the de Rham spectrum $\clS X$ is a smooth superscheme of finite type,
and the sheaf $\Omega^1_{\clS X}$ is locally described as 
\begin{align*}
 \rst{\Omega_{\clS X}^1}{U} \cong 
 \bigoplus_{i=1}^n \shO_{\clS U} \cdot dx_i \oplus 
 \bigoplus_{i=1}^n \shO_{\clS U} \cdot d\xi_i,
\end{align*}
where $U \subset X$ is an open subscheme equipped with a local coordinate system $x_1,\dotsc,x_n$,
and $\xi_i \in \shO_{\clS X}(U)_{\od}$ is the element corresponding to 
$dx_i \in \Omega_X^1(U) =  \Omega_X^1(U)_{\ev}$ under the parity change $\Pi$..
\end{lem}

\subsection{Formal superloop space and the chiral de Rham complex}

Let $T=(t,\tau)$ be a pair of an even determinate $t$ and an odd determinate $\tau$,
which will be called a \emph{supervariable}. 
We denote the polynomial superalgebra of supervariable $T$ over $\bbC$ by 
\begin{align}\label{eq:L:CT}
 \bbC \dbr{T} \ceq \bbC\dbr{t}[\tau] = \bbC\dbr{t} \otimes_{\bbC} \Lambda[\tau].
\end{align}

Let $X$ be a superscheme.
We denote by $\clJ\clS$ the composition of the functors $\clJ$ and $\clS$.
By \cite[Proposition 4.2.7]{KV11}, the superscheme $\clJ\clS X$ represents the functor
\begin{align*}
 \SSch \lto \cSet, \quad S \lmto \Hom_{\SSp}((\ul{S},\shO_S\dbr{T}), X),
\end{align*}
where $\shO_S\dbr{T} \ceq \shO_S \otimes_{\bbC} \bbC\dbr{T}$.
In particular, the functors $\clJ$ and $\clS$ commute: $\clJ\clS X = \clS\clJ X$.
We have a morphism of superschemes
\begin{align}\label{eq:L:pJS}
 p_{\clJ\clS}\colon \clJ\clS X \lto X.
\end{align}
We call $\clJ\clS X$ the \emph{formal superjet space} of $X$.

Next, let $X$ be a superscheme of finite type. 
We denote by $\LS$ the composition of the functors $\clL$ and $\clS$.
By \cite[Proposition 4.2.7]{KV11}, the ind-superscheme $\LS X$ represents the functor 
\begin{align}\label{eq:L:pLS}
 \SSch \lto \cSet, \quad S \lmto \Hom_{\SSp}((\ul{S},\shO_S\dpr{t}^{\sqr}[\tau]),X),
\end{align}
and equipped with a morphism
\[
 p_{\LS}\colon \LS X \lto X.
\]
We call $\LS X$ the \emph{formal superloop space} of $X$.

We recall some basic properties of $\JS X$ and $\LS X$.

\begin{lem}\label{prp:L:LSNX}
Let $X$ be a superscheme of finite type. 
\begin{enumerate}
\item 
The ind-superscheme $\LS X$ is an inductive limit of nilpotent extensions 
of the superscheme $\JS X$.

\item
We have a diagram of ind-superschemes 
\begin{equation}\label{diag:L:SJSLS}
\begin{tikzcd}
 X & \clS X \ar[l,"p_{\clS}"'] &
 \clJ\clS X \ar[l,"p_{\clJ}"'] \ar[r,"i"] & \LS X 
\end{tikzcd}
\end{equation}
where $p_{\clJ}$ and $p_{\clS}$ are given by \eqref{eq:L:pJ} and \eqref{eq:L:pS}.
We also have $p_{\clS}p_{\clJ}=p_{\JS}=p_{\LS}i\colon \JS X \to X$.

\item
For an open subscheme $U \subset X$, we have $\LS U \cong \rst{\LS X}{p_{\LS}^{-1}(U)}$.
\end{enumerate}
\end{lem}

\begin{proof}
The statements follows from \cite[Theorem 1.4.2]{KV04} and \cite[\S4.3]{KV11}.
\end{proof}

Recall the definition of the structure sheaf $\shO_Y$ of an ind-superscheme $Y$ \cite[\S3.2]{KV11}.
In particular, we have the structure sheaf $\shO_{\LS X}$ of the formal superloop space $\LS X$.
By the discussion after \eqref{eq:L:vpO=Om}, 
${p_{\LS}}_*\shO_{\LS X}$ is a sheaf of (super)complexes on $X$.
Then, we have: 

\begin{prp}\label{prp:L:OLS}
For a smooth scheme $X$ of finite type over $\bbC$, we have an isomorphism
\[
 {p_{\LS}}_* \shO_{\LS X} \cong \Omega_X^{\ch}
\]
of sheaves of complexes on $X$, 
where $\Omega_X^{\ch}$ is the \emph{chiral de Rham complex} of $X$ 
(see \cite{MSV} and \cite[5.3]{KV04}).
\end{prp}


\begin{proof}
By \cite[\S3]{KV04}, for the (ind-)superschemes $\JS X$ and $\LS X$,
we have the notion of left and right $\shD$-modules on them,
and the categories $\Dqcr{\JS X}$ and $\Dqcr{\LS X}$ of quasi-coherent right $\shD$-modules.
Moreover, for each $\shM \in \Dqcr{X}$, 
we have well-defined objects $p_{\clJ\clS}^*\shM \in \Dqcr{\clJ\clS X}$ and 
$i_*p_{\clJ\clS}^*\shM \in \Dqcr{\LS X}$ associated to the diagram \eqref{diag:L:SJSLS}.
Then, since $\LS U$ is a nilpotent-extension of $\JS U$ and 
$\clL\clS U \cong \rst{\clL\clS X}{p_{\clJ\clS}^{-1}(U)}$ for an open subset $U \subset X$ 
(\cref{prp:L:LSNX}), we have a well-defined right $\shD$-module 
$\rst{i_*p_{\clJ\clS}^*\shM}{\clL\clS U}$ on $\LS U$.
Now we define $\CDR(\shM)$ to be the the complex of sheaves of linear spaces on $X$ 
by the correspondence 
\begin{align*}
 \CDR(\shM)\colon U \lmto \Sp(\rst{i_*p_{\clJ\clS}^*\shM}{\clL\clS U}),
\end{align*}
where, for a right $\shD$-module $\shN$ on the ind-superscheme $\clL\clS U$, $\Sp(\shN)$ denotes
the de Rham complex $\operatorname{DR}(\shN)$ constructed in \cite[\S4.3]{KV04}
(we use the symbol $\Sp$ and call it the Spencer complex, 
 following the standard terminology \cite[Lemma 1.5.27]{HTT}). 
In the case $\shM=\omega_X$, 
the sheaf of differential forms of top degree on the smooth scheme $X$, we denote
$\CDR_X \ceq \CDR(\omega_X)$.
Now, \cite[5.3.1]{KV04} states that 
\begin{align*}
 \CDR_X \cong \Omega_X^{\ch}
\end{align*} 
as sheaves of complexes on $X$.
Turning this ``right'' $\shD$-module'' argument the ``left'' one, we have the consequence.
\end{proof}

\subsection{Global formal superloop space} 

We recall the notion of a global formal loop space in the super setting 
from \cite[\S2.1]{KV04} and \cite[\S5.1]{KV11}.

Let $C$ be a smooth supercurve (of dimension $1|1$).
For a non-empty finite set $I$ and a superscheme $S$, we express a morphism $c_I\colon S \to C^I$
as a tuple $c_I=(c_i)_{i \in I}$ with $c_i\colon S \to C$.
We denote by $\Gamma=\Gamma(c_I) \subset S \times C$ the union of the graphs 
of $c_i\colon S \to C$, $i \in I$.
Then, we have the following sheaves of superalgebras on $S \times C$ supported on $\Gamma$.
\begin{itemize}
\item 
$\wh{\shO}_\Gamma$, the completion of $\shO_{S \times C}$ along $\Gamma$.
Thus, it is the sheaf of functions on the formal neighborhood of $\Gamma$.

\item
$\shK_\Gamma \ceq \wh{\shO}_\Gamma[r^{-1}]$, where $r$ is a local equation of $\Gamma$ in $S \times C$.
Thus, it is the sheaf of functions on the punctured formal neighborhood of $\Gamma$.

\item
$\shK_\Gamma^{\sqr} \subset \shK_\Gamma$, the subsheaf of sections whose restriction to 
$S_{\tred} \times C$ (see \eqref{eq:SC:red}) lies in $\wh{O}_{\Gamma_{\tred}}$.
Here $\Gamma_{\tred}$ is the graph of the restriction $c_{I,\tred}\colon S_{\tred} \to C^I$ 
of $c_I$ to $S_{\tred}$.
\end{itemize}

Now, let $X$ be a superscheme of finite type.
We define two functors $\lambda_{X,C^I},j_{X,C^I}\colon \SSch \to \cSet$ by
\begin{align}
\nonumber
 j_{X,C^I}&\colon S \lmto \{(c_I,\phi) \mid c_I\colon S \to C^I, \, 
 \phi \in \Hom_{\SSp}((\ul{\Gamma},\wh{\shO}_\Gamma),X)\}, \\
\label{eq:L:lamXCI}
 \lambda_{X,C^I}&\colon S \lmto \{(c_I,\phi) \mid c_I\colon S \to C^I, \, 
 \phi \in \Hom_{\SSp}((\ul{\Gamma},\shK_\Gamma^{\sqr}),X)\}.
\end{align}

\begin{fct}[{\cite[\S2.5, \S2.6]{KV04}, \cite[\S5.3]{KV11}}]\label{fct:L:JCI-LCI}
Let $X,C,I$ be as above.
There is a superscheme $\JC{I}X$ represents the functor $j_{X,C^I}$, 
and an ind-superscheme $\LC{I}X$ representing $\lambda_{X,C^I}$.
We call them the global formal superjet space and the global formal superloop space
(for $I$), respectively.
\end{fct}

We have a diagram similar to \eqref{diag:L:SJSLS}:
\begin{equation*}
\begin{tikzcd}
 X \times C^I & \JC{I}X \ar[l,"p"'] \ar[r,"i"] & \LC{I}X 
\end{tikzcd}
\end{equation*}
Here $p$ is the projection $(c_I,\phi) \mto \Gamma(c_I)$, and $i$ is a closed embedding.
If $I$ consists of one point, then
we denote these (ind-)superschemes by $\clJ\clS_CX$ and $\clL\clS_CX$.
They can be described explicitly using Harish-Chandra pairs, as we will recall below.

We have the notion of a \emph{Harish-Chandra pair} $(\frg,K)$ in the sense of \cite[2.9.7]{BD}.
It is a pair of a (topological) Lie algebra $\frg$ and an affine group scheme $K$
equipped with a (continuous) embedding of Lie algebras $\phi\colon \frk \ceq \Lie K \inj \frg$ 
and a $K$-action $\alpha$ on the Lie algebra $\frg$ such that the $\frk$-action on $\frg$ 
induced by $\alpha$ is equal to $\ad_{\phi}$.
See also \cite[\S17.2]{FB} for explanations.
We have an obvious super version of Harish-Chandra pairs.

Let $\bbC[Z]$ be the polynomial superalgebra over $\bbC$ with supervariable $Z=(z,\zeta)$ 
(see \eqref{eq:L:CT}).
As explained in \cite[\S5.3]{KV11}, we have a Harish-Chandra pair 
\begin{align}\label{eq:L:gK}
 (\frg_{1|1},K_{1|1}), \quad 
  \frg_{1|1} \ceq \Der_{\bbC}(\bbC\dbr{Z}), \quad 
  K_{1|1} \ceq \Aut(\bbC\dbr{Z}).
\end{align}
of the (topological) Lie superalgebra $\frg_{1|1}$ of derivations and 
the group superscheme $K_{1|1}$ of automorphisms of the supercommutative algebra $\bbC\dbr{Z}$.
The derivations $\pd_Z=(\pd_z,\pd_\zeta)$ (see \cref{eq:SC:pdZ})
form a (topological) generator of $\frg_{1|1}$.
For a supercommutative ring $R$, an $R$-point of $K_{1|1}$ is an invertible change of coordinates
\[
  z    \lmto \sum_{i \ge 0, \, j \in \{0,1\}} a_{i,j} z^i \zeta^j, \quad 
 \zeta \lmto \sum_{i \ge 0, \, j \in \{0,1\}} b_{i,j} z^i \zeta^j.
\]
Here the element $a_{i,j} \in R$ is of parity $\ol{j}$, and
the element $b_{i,j} \in R$ is of parity $\ol{j+1}$.
We require that $a_{0,0}=0$, $a_{1,0} \in R_{\ev}^{\times}$, and $b_{0,1} \in R_{\ev}^\times$.
Note that the Lie superalgebra $\frk_{1|1} \ceq \Lie(K_{1|1})$ is the Lie sub-superalgebra of 
$\frg_{1|1}$ which preserve the maximal ideal $\frm \subset \bbC\dbr{Z}$,
and is generated by $z\pd_z$ and $\pd_\zeta$.

Let $C$ be a smooth supercurve of dimension $1|1$.
Following \cite[\S5.3]{KV11}, we denote by 
\begin{align}\label{eq:GL:whC}
 \pi\colon \wh{C} \lto C
\end{align}
the superscheme over $C$ whose $S$-points for a superscheme $S$ are tuples $(c,Z)$ 
of an $S$-point $c\colon S \to C$ and a local coordinate system 
$Z=(z,\zeta)$ at $c$ (see \cref{eq:SC:Z}).
The projection $\pi$ is given by $(c,Z) \mto c$.

The superscheme $\wh{C}$ over $C$ has a natural $(\frg_{1|1},K_{1|1})$-structure \cite[2.9.8]{BD},
i.e., compatible actions of $\frg_{1|1}$ and $K_{1|1}$ such that 
$K_{1|1}$ acts along $\pi$, making $\wh{C}$ a $K_{1|1}$-torsor over $C$,
and the action of $\frg_{1|1}$ is (formally) simply transitive.
The latter condition is equivalent to the one that the morphism 
$(\frg_{1|1}/\frk_{1|1}) \otimes \shO_C \to \pi^*\Theta_C$ defined by the $\frg$-action 
and $d\pi$ is an isomorphism.

Now, let $X$ be a superscheme of finite type,
so that we have the formal superjet space $\JS X$ \eqref{eq:L:pJS} 
and the formal superloop space $\LS X$ \eqref{eq:L:pLS}.
By construction, the group $K_{1|1}$ acts fiberwise on 
the structure morphisms $p_{\JS}\colon \JS X \to X$ and $p_{\LS}\colon \LS X \to X$. 
Hence, we can construct a superscheme and an ind-superscheme
\begin{align}\label{eq:L:JC-LC}
 \clJ\clS X \times_{K_{1|1}} \wh{C}, \quad \clL\clS X \times_{K_{1|1}} \wh{C}.
\end{align}
By construction, we have the structure morphisms from these (ind-)superschemes to $X \times C$.
Let us cite:

\begin{fct}[{\cite[2.1.2]{KV04}, \cite[\S5.1, \S5.3]{KV11}}]\label{fct:L:LC=LS}
Let $X$ be a superscheme of finite type, and $C$ be a smooth superscheme of dimension $1|1$.
Then we have isomorphisms 
\[
 \JC{}X \cong \clJ\clS X \times_{K_{1|1}} \wh{C}, \quad 
 \LC{}X \cong \clL\clS X \times_{K_{1|1}} \wh{C}.
\]
\end{fct}

Similarly to \cref{prp:L:OLS}, we have an explicit description 
of the structure sheaf $\shO_{\LC{}X}$. 

\begin{prp}\label{prp:FS:OLSC}
Let $X$ and $C$ be as in \cref{fct:L:LC=LS}, 
and denote by $p\colon \LC{}X \to X \times C$ the projection.
Then the fiber of $p_*\shO_{\LC{}X}$ at a point of $C$ is isomorphic to 
${p_{\LS}}_*\shO_{\LS X}$, and hence isomorphic to $\Omega_X^{\ch}$ by \cref{prp:L:OLS}.
\end{prp}

\begin{proof}
This is a rewritten form of \cite[5.1.2.\ Proposition (a)]{KV04}, 
setting the right $\shD_X$-module $\shM$ to be $\omega_X$. 
\end{proof}

We refer to \cite[Proposition 3.1.7]{KV11} 
for the notion of an integrable connection on $E$ along $B$,
where $B$ is a superscheme and $E$ is an ind-superscheme $B$.
It is a natural extension of the crystalline description of a left $\shD$-module 
on a smooth algebraic variety (see \cite[Appendix]{G} and \cite[\S8]{S} for the non-super case).

%

which makes sense not only for superschemes but for ind-superschemes.
By the standard argument (see \cite{BB} and \cite[\S17.2]{FB} for example), 
the $(\frg_{1|1},K_{1|1})$-structure on $\wh{C} \to C$ induces an integrable connection on $\wh{C}$
along $C$, given by the action of the generators $\pd_Z=(\pd_z,\pd_\zeta)$ of $\frg_{1|1}$.
It induces integrable connections on the superspaces \eqref{eq:L:JC-LC} along $C$.
Thus, by \cref{fct:L:LC=LS}, 
the superspaces $\JC{}X$ and $\JC{}X$ have integrable connections along $C$.
More generally, we have:

\begin{fct}[{\cite[Proposition 5.2.7]{KV11}}]\label{fct:L:ic}
For each non-empty finite set $I$, 
the global formal superjet space $\JC{I}X$ and the global formal superloop space $\LC{I}X$
have integrable connections along $C$.
\end{fct}

\section{Factorization structure}\label{s:FS}

We recall the factorization structure on 
the global superloop spaces $\LC{I}X$ from \cite{KV04,KV11}.
We continue to work over $\bbC$.

\subsection{Factorization space}\label{ss:FS:FS}

We recall several terminology from \cite[\S2.2]{KV04} and \cite[\S5.2]{KV11}.

\begin{dfn}
Let $C$ be a superscheme, 
and $p\colon J \srj I$ be a surjection of non-empty finite sets.
\begin{enumerate}
\item
For $i \in I$, we denote $J_i \ceq p^{-1}(i) \subset J$.

\item
We define the \emph{diagonal map associated with $p$} to be the embedding of superscheme
\begin{align}\label{eq:L:Delta}
 \Delta_p\colon C^I \linj C^J, \quad 
 (c_i)_{i \in I} \lmto (c_{p(j)})_{j \in J}.
\end{align}

\item
We denote an open sub-superscheme $U_p \subset C^J$ to be 
the complement of the partial diagonals, i.e., 
\[
 U_p \ceq \{(c_j)_{j \in J} \in C^J \mid 
            \text{$c_j \ne c_k$ for any pair $j,k \in J$ with $p(j) \ne p(k)$}\}.
\]
We denote the associated embedding by
\begin{align}\label{eq:L:jp}
 j_p\colon U_p \linj C^J.
\end{align}
\end{enumerate}
\end{dfn}

\begin{rmk}
The notational comments are in oder.
\begin{enumerate}
\item
Our $\Delta_p\colon C^I \to C^J$ is denoted
by $\Delta^{(p)}$ or $\Delta^{(J/I)}$ in \cite[2.2.3]{BD} and \cite[2.2]{KV04}. 


\item
Our $j_p\colon U_p \inj C^J$ is denoted
by $j^{(J/I)}\colon U^{(J/I)} \inj C^J$ in \cite[2.2]{KV04},
by $j^{(p)}\colon U^{(p)} \inj C^J$ in \cite[3.1.2]{BD}, 
by $j^{[J/I]}\colon U^{[J/I]} \inj C^J$ in \cite[3.4.4]{BD},
and by $\jmath(p)\colon U(p) \inj C^J$ in \cite[2.2.3]{FG}.
\end{enumerate}
\end{rmk}

\begin{eg}
Let $C$ be a superscheme, and $J$ be a non-empty finite set.
\begin{enumerate}
\item 
Let $p_J\colon J \to \{1\}$ be the surjection from $J$ to the one point set. Then 
\begin{align}\label{eq:L:Delta_J}
 \Delta_J \ceq \Delta_{p_J}\colon C \linj C^J
\end{align}
is the small diagonal map.

\item
Let $\id_J\colon J \to J$ be the identity map of $J$. Then 
$U_J \ceq U_{\id_J} \subset C^J$ is the complement of the big diagonal 
$\{(c_j)_{j \in J} \in C^J \mid \text{$c_j=c_k $ for some $j \ne k \in J$}\}$.
The associated embedding $j_{id_J}$ is denoted by
\begin{align}\label{eq:L:j_J}
 j_J\colon U_J \linj C^J
\end{align}
\end{enumerate} 
\end{eg}

Hereafter, using the terminology of schemes, 
an ind-superscheme over a superscheme $C$ is called a \emph{$C$-ind-superscheme}.
Similarly we use the words \emph{$C$-morphism} and \emph{$C$-isomorphisms} of $C$-ind-superschemes.

\begin{dfn}[{\cite[2.2.1]{KV04}}]\label{dfn:FS:FS}
A \emph{factorization space} over a smooth superscheme $C$ is the following data.
\begin{itemize}
\item
For each non-empty finite set $J$, 
an ind-superscheme $Y_J$ over $C^J$ which is formally smooth \cite[\S3.1, p.1095]{KV11}
and equipped with integrable connections along $C^J$ \cite[Proposition 3.1.7]{KV11}.

\item
For each surjection $p\colon J \srj I$, 
a $C^I$-isomorphism $\nu_p\colon \Delta_p^*(Y_J) \sto Y_I$
which respects the integrable connections.

\item
For each pair $K\sr{q}J\sr{p}I$ of composable surjections, 
a $C^p$-isomorphism $\kappa_p\colon j_p^*(\prod_{i \in I}Y_{J_i}) \sto j_p^*(Y_J)$
which respects the integrable connections.
\end{itemize}
These should satisfy the following the following conditions 
for each pair $K\sr{q}J\sr{p}I$ of composable surjections.
\begin{gather*}
 \nu_{pq}=\nu_p\circ\Delta_p^*(\nu_q), \\
 \kappa_q =  \kappa_{pq} \circ \tprd_{i \in I} \kappa_{q|K_i}, \\
 \nu_q \circ \Delta_q^*(\kappa_{pq}) = \kappa_p \circ \tprd_{i \in I}\nu_{q|K_i}.
\end{gather*}
Here $q|K_i\colon K_i \to J_i$ is the restriction of $q$ to $K_i \ceq (pq)^{-1}(i) \subset K$.
\end{dfn}

We denote a factorization space as $(Y_I)_I$ for simplicity.
The isomorphisms $\kappa_p$ are called the \emph{factorization isomorphisms} of $(Y_I)_I$.

\begin{rmk}
We comment on the terminology and notations.
\begin{enumerate}
\item 
Our factorization space is called a factorization semi-group in \cite{KV11}.
See \cite[Remark 5.2.5]{KV11} for the related notions appearing in \cite{BD} and \cite{KV04}.

\item
Our $\nu_p$ is denoted by $\nu^{(J/I)}$ in \cite[2.2.1]{KV04}, and corresponds to 
$\varkappa_{p_I,p}$ with $p_I\colon I \srj \{1\}$ in \cite[Definition 5.2.2 (b)]{KV11}.

\item
Our $\kappa_p$ for $p\colon J \srj I$ is denoted by $\kappa^{(J/I)}$ in \cite[2.2.1]{KV04},
and corresponds to $c_{[J/I]}$ in \cite[3.4.4]{BD}.
There seems to be a typing error in the corresponding condition of \cite[2.2.1 (b)]{KV04}.
\end{enumerate}
\end{rmk}


Now, let $X$ be a superscheme of finite type, and $C$ be a smooth superscheme of dimension $1|1$.
By \cref{fct:L:JCI-LCI}, we have the family $(\LC{I}X)_I$ over $X \times C^I$, 
where $I$ runs over non-empty finite sets.
We denote the composition of the structure morphism $\LC{I}X \to X \times C^I$ and 
the projection $X \times C^I \to C^I$ by
\begin{align}\label{eq:FS:pI}
 p_I\colon \LC{I}X \lto C^I.
\end{align} 
Also, by \cref{fct:L:ic}, there is an integrable connections on $\LC{I}X$ along $C$. 
Then we have: 

\begin{fct}[{\cite[2.3.3]{KV04}, \cite[Proposition 5.3.2]{KV11}}]\label{fct:FS:LSC}
Let $X$ and $C$ be as above. Then
the family $(\LC{I}X)_I$ 
has a structure of a factorization space over $C$.
\end{fct}

For later reference, let us outline the structure.
Let $p\colon J \srj I$ be a surjection of non-empty sets.
Recall the functor $\lambda_{X,C^J}$ represented by $\LC{J}X$ in \eqref{eq:L:lamXCI}.
As for the isomorphism $\nu_p$ in \cref{dfn:FS:FS}, note that 
an $S$-point (where $S$ is a superscheme) of $\Delta_p^*(\LC{J}X)$ 
is a pair $(c_J,\phi)$ where $c_J$ is a morphism $S \to C^J$ 
whose image lies in the partial diagonal $\Delta_p \cong C^I \subset C^J$.
Thus, $c_J$ comes from a uniquely defined $I$-tuple $c_I\colon S \to C^I$.
Then, $\Gamma(c_J)=\Gamma(c_I)$, and the sheaves $\shK^{\sqr}$ associated to them coincide.
Now, we have a natural identification of $S$-points of $\Delta_p^*(\LC{J}X)$ and $\LC{I}X$,
and hence we have the isomorphism $\nu_p$.
As for the factorization isomorphism $\kappa_p$, 
an $S$-point of $j_p^*(\LC{J}X)$ is a pair $(c_J,\phi)$ with $c_J\colon S \to U_p$.
Then we have $\Gamma(c_J)=\bigsqcup_{i \in I}\Gamma(c_{J_i})$ with $J_i \ceq p^{-1}(i) \subset J$.
Hence, we have an isomorphism 
$\shK^{\sqr}_{\Gamma(c_J)}[\tau] \cong \prod_{i \in I}\shK^{\sqr}_{\Gamma(c_{J_i})}[\tau]$,
which induces the isomorphism $\kappa_p$.

\subsection{Factorization \texorpdfstring{$\shD$}{D}-module}

We recall the notion of a factorization algebra \cite[3.4]{BD} in the super setting.
Recall the notation of $\shD$-modules given in \cref{ss:SC:pre}.

\begin{dfn}[{c.f.\ \cite[3.4.1, 3.4.4]{BD}}]\label{dfn:FS:FA}
Let $C$ be a smooth supervariety.
A \emph{factorization $\shD$-module over $C$} is the following data.
\begin{itemize}
\item 
For each non-empty set $J$, an object $\shF_J \in \Dqcl{C^J}$
which has no non-zero local sections supported at the union of all partial diagonals
$\Delta_p\colon C^I \inj C^J$, $p\colon J \srj I$ (see \eqref{eq:L:Delta}).

\item
For each surjection $p\colon J \srj I$ of non-empty finite sets, an isomorphism 
$\nu_p\colon \Delta_p^*\shF_J \sto \shF_I$ in $\Dqcl{C^I}$. 

\item
For each $p\colon J \srj I$, an isomorphism 
$\kappa_p\colon j_p^*(\boxtimes_{i \in I}\shF_{J_i}) \sto j_p^*(\shF_J)$ in $\Dqcl{C^J}$,
where $j_p\colon U_p \inj C^J$ is the open embedding \eqref{eq:L:jp}.
\end{itemize}
These should satisfy the following conditions for each composable pair $K\sr{q}J\sr{p}I$.
\begin{gather*}
 \nu_{pq}=\nu_p \circ \Delta_p^*(\nu_q)\colon \Delta_{pq}^*\shF_K \lsto \shF_I, \\ 
 \kappa_q =  \kappa_{pq} \circ \tprd_{i \in I} \kappa_{q|K_i}, \\
 \nu_q \circ \Delta_q^*(\kappa_{pq}) = \kappa_p \circ \tprd_{i \in I}\nu_{q|K_i}.
\end{gather*}
Here $q|K_i\colon K_i \to J_i$ is the restriction of $q$ to $K_i \ceq (pq)^{-1}(i) \subset K$.
\end{dfn}

We denote a factorization $\shD$-module as $(\shF_I)_I$ for simplicity.
A morphism of factorization $\shD$-modules over $C$ is naturally defined.

\begin{rmk}
Some comments are in order.
\begin{enumerate}

\item
In \cite{BD}, our $\shF_I$ is denoted by $\shF_{C^I}$,  
our $\nu_p$ is denoted by $\nu^{(p)}=\nu^{(J/I)}$, 
and our $\kappa_p$ is denoted by $c_{[J/I]}$.

\item
Our terminology ``factorization $\shD$-module" respects the proposal in \cite[Remark 2.4.8]{FG}. 
It is called a factorization algebra in \cite{BD}.
\end{enumerate}
\end{rmk}

Recall the definition of the structure sheaf $\shO_Y$ of an ind-superscheme $Y$ \cite[\S3.2]{KV11}.
By comparing \cref{dfn:FS:FS} and \cref{dfn:FS:FA}, we immediately have:

\begin{lem}\label{lem:FS:p*O}
Let $C$ be a smooth superscheme, and $(Y_I)_I$ be a factorization space over $C$.
For each non-empty finite set $I$, denote the structure morphism $Y_I \to C^I$ by $p_I$.
Then, the family $\bigl((p_I)_*\shO_{Y_I}\bigr)_I$ of sheaves on $C^I$'s 
has a structure of a factorization $\shD$-module over $C$.
\end{lem}

In particular, by \cref{fct:FS:LSC}, we have: 

\begin{prp}\label{prp:FS:pLC}
Let $C$ be a smooth supercurve of dimension $1|1$, and $X$ be a smooth scheme of finite type.
Then the family $\bigl({p_I}_*\shO_{\LC{I}X}\bigr)_I$ has a structure of a factorization 
$\shD$-module over $C$, where $p_I\colon \LC{I}X \to X$ is given in \eqref{eq:FS:pI}.
\end{prp}

\subsection{Equivalence with chiral algebras}\label{ss:L:NW}

By \cite[3.4]{BD} and \cite{FG}, a factorization $\shD$-module on a smooth curve is 
equivalent to a chiral algebra \cite[3.1]{BD} on the curve. 
We explain that the equivalence is naturally extended to the super setting,
using the chiral algebra on a smooth supercurve introduced in \cite[\S4]{H}.


Let $C$ be a smooth supercurve of dimension $1|1$.
We denote the diagonal embedding and the open embedding of the complement by 
$\Delta\colon C \linj C^2$, $j\colon U \linj C^2$ as in \eqref{eq:SC:D-U}
(denoted by $\Delta_{\{1,2\}}$ in \eqref{eq:L:Delta_J} and 
 by $j_{\{1,2\}}\colon U_{\{1,2\}} \inj C^2$ in \eqref{eq:L:j_J}).
We also denote by $\Delta_3\colon C \inj C^3$ the small diagonal embedding
($\Delta_{\{1,2,3\}}$ in \eqref{eq:L:Delta_J}).
The open embedding of the complement of the big diagonal is denoted by $j_3\colon U_3 \inj C^3$
($j_{\{1,2,3\}}\colon U_{\{1,2,3\}} \inj C^3$ in \eqref{eq:L:j_J}).

Let $\shA \in \Dqcr{C}$, i.e., a quasi-coherent right $\shD_C$-module, and 
$\mu\colon j_*j^*(\shA^{\boxtimes 2}) \to \Delta_*\shA$ be a morphism in $\Dqcr{C^2}$.
We define $\mu_{1\{23\}}\colon {j_3}_*j_3^*(\shA^{\boxtimes 3}) \lto {\Delta_3}_*\shA$
by first applying $\mu$ to the second and third factor, and then applying $\mu$ to 
the first factor and the result. Similarly we define $\mu_{\{12\}3}$. 
We also define $\mu_{2\{13\}}$ using the transposition $\sigma_{12}$ of the first and second factors.
See \cite[19.3.1]{FB} and \cite[4.1.10]{H} for the detail of these definitions.

\begin{dfn}[{\cite[Definition 4.10]{H}}]
Let $C$ be a smooth supercurve of dimension $1|1$.
A \emph{chiral algebra} on $C$ is a quasi-coherent right $\shD_C$-module $\shA$ 
equipped with a morphism $\mu\colon j_*j^*(\shA^{\boxtimes 2}) \to \Delta_*\shA$ 
in $\Dqcr{C^2}$ satisfying the following conditions.
\begin{clist}
\item
$\mu=-\mu \circ \sigma_{12}$.

\item
$\mu_{1\{23\}}=\mu_{\{12\}3}+\mu_{2\{13\}}$
\end{clist}
We call $\mu$ the \emph{chiral bracket} of the chiral algebra $\shA$.
\end{dfn}

We also have the notion of a \emph{morphism} 
and the \emph{category of chiral algebras} on $C$.

%

Recall that a factorization $\shD$-module over $C$ (\cref{dfn:FS:FA}) is a family of 
quasi-coherent left $\shD_{C^I}$-modules $\shF_I$ for non-empty finite sets $I$
equipped with compatible system of isomorphisms.
The following categorical equivalence is shown for 
the non-super case in \cite[3.4.9]{BD} and \cite{FG}
(precisely speaking, \cite{BD} gives the equivalence for unital objects).
The proof works also for the super case, as mentioned in \cite[3.4.9, Remark (ii)]{BD}, and we have:

\begin{fct}\label{fct:FS:FA=CA}
Let $C$ be a smooth supercurve of dimension $1|1$.
Then, we have an equivalence between
\begin{itemize}
\item 
the category of factorization $\shD$-modules $(\shF_I)_I$ over $C$, 
\item
the category of chiral algebras $\shA$ on $C$,
\end{itemize}
which lifts the correspondence of left and right $\shD$-modules 
$\shF_{\{1\}} \mto \shA = (\shF_{\{1\}})^r$ (see \eqref{eq:D:L=R}).
\end{fct}

For later reference, we recall from \cite[3.4.8]{BD} the chiral bracket $\mu$ 
corresponding to a factorization $\shD$-module $(\shF_I)_I$.
We denote $\shF_1 \ceq \shF_{\{1\}}$ and $\shF_2 \ceq \shF_{\{1,2\}}$ for simplicity.
For the surjection $p_{[2]}\colon \{1,2\} \srj \{1\}$, 
the isomorphism $\nu_{p_{[2]}}$ induces a right $\shD$-module isomorphism
$\nu^r\colon (\Delta_*\omega_C) \otimes \shF_2 \sto \Delta_*(\omega_C \otimes \shF_1)=\Delta_*\shA$.
Recall the morphism $\mu_\omega$ in \eqref{eq:SC:mu-omega}, which can be written 
using \eqref{eq:SC:jj} as $\mu_\omega\colon j_*j^*(\omega_C^{\boxtimes 2}) \to \Delta_*\omega_C$, 
where $j\colon U \inj C^2$ is the open embedding of the complement of $\Delta \subset C^2$
(see \eqref{eq:SC:D-U}).
Then the composition
\begin{align}\label{eq:FS:mu}
 j_*j^*\shA = j_*j^*\bigl((\omega_C \otimes \shF_1)^{\boxtimes 2}\bigr) 
 \xr[\id \otimes \kappa_{p_{[2]}}]{\sim}
 \bigl(j_*j^*(\omega_C^{\boxtimes 2})\bigr) \otimes \shF_2
 \xr[\mu_\omega \otimes \id]{} (\Delta_*\omega_C) \otimes \shF_2
 \xr[\nu^r]{\ \sim \ }
 \Delta_*\shA
\end{align}
is the desired chiral bracket on $\shA$.

Now recall from \cref{prp:FS:pLC} that we have a factorization $\shD$-module
$\bigl({p_I}_*\shO_{\LC{I}X}\bigr)_I$ for a smooth scheme $X$ of finite type. 
The underlying right $\shD_C$-module of the corresponding chiral algebra in \cref{fct:FS:FA=CA} 
is $\shF^r$, where $\shF \ceq {p_{\{1\}}}_*\shO_{\LC{}X}$.
By \cref{prp:FS:OLSC}, the fiber of $\shF$ over a point of $C$ 
isomorphic to the chiral de Rham complex $\Omega_X^{\ch}$.
Let us summarize the argument as:

\begin{cor}\label{cor:FS:CdR}
Let $C$ be a smooth supercurve of dimension $1|1$, and $X$ be a smooth scheme of finite type.
Then there is a chiral algebra on $C$ whose fiber at a point of $C$ is isomorphic to 
the chiral de Rham complex $\Omega_X^{\ch}$.
\end{cor}

\section{\texorpdfstring{$N_K=1$}{NK=1} SUSY structure from factorization structure}\label{s:NK}

We write down the chiral algebra structure on $\Omega_X^{\ch}$ (\cref{cor:FS:CdR})
in the case $C$ is a superconformal curve, and show that 
it gives a $N_K=1$ SUSY vertex algebra structure on $\Omega_X^{\ch}$.

First, we introduce natural superconformal analogue of the materials in \cref{ss:FS:FS}.
Let $(C,\shS)$ be a superconformal curve, and $Z=(z,\zeta)$ be its superconformal coordinate system.
For a non-empty finite set $J$, let $Z_j \ceq p_j^*Z$ ($j \in J$) be the induced coordinate system 
on $C^J$ with $p_j\colon C^J \to C$ the $j$-th projection.

For a surjection $p\colon J \srj I$ of non-empty finite sets, 
we define $U^{\tsc}_p \subset C^J$ the open sub-superscheme which is locally
defined by $z_j-z_k-\zeta_j\zeta_k \ne 0$ for $j,k \in J$ with $p(j) \ne p(k)$.
It is also well defined. The embedding is denoted by 
\[
 j^{\tsc}_p\colon U^{\tsc}_p \linj C^J.
\]
For the surjection $p_{[2]}\colon \{1,2\} \srj \{1\}$, then $U^{\tsc}_{p_{[2]}} \subset C^2$ 
is the complement of the superdiagonal $\Delta^{\tsc} \subset C^2$.

Now, let $X$ be a smooth algebraic variety, 
and consider the factorization space $(\JC{I}X)_I$ in \cref{fct:FS:LSC}.
Let us denote $Y_I \ceq \JC{I}X$ for simplicity.
Then, we have:

\begin{lem}\label{lem:NK:kappa}
For each surjection $p\colon J \srj I$ of non-empty finite sets, 
the factorization isomorphism  $\kappa_p$ of $(Y_I=\JC{I}X)_I$ induces 
\[
 \kappa_p^{\tsc}\colon {j_p^{\tsc}}^*(\tprd_{i \in I}Y_{J_i}) \lsto {j_p^{\tsc}}^*(Y_J).
\]
\end{lem}

\begin{proof}
Recall the argument on $\kappa_p$ explained in the paragraph after \cref{fct:FS:LSC}.
For a superscheme $S$, an $S$-point of $j_p^*(\LC{J}X)$ is a pair $(c_J,\phi)$
with $c_J\colon S \to U_p^{\tsc}$.
Hence, we have $\Gamma(c_J)=\bigsqcup_{i \in I}\Gamma(c_{J_i})$, and 
the sheaves $\shK^{\sqr}$ associated to them are isomorphic.
This shows the equality of $S$-points of ${j_p^{\tsc}}^*(Y_J)$ and 
${j_p^{\tsc}}^*(\tprd_{i \in I}Y_{J_i})$, which yields the isomorphism $\kappa_p^{\tsc}$.
\end{proof}

We denote by $p_J\colon Y_J \to C^J$ the structure morphism. 
Then, $\shF_J \ceq {p_J}_*\shO_{Y_J}$ for non-empty finite sets $J$ form 
a factorization $\shD$-module by \cref{lem:FS:p*O}.
Hence, we have an isomorphism of left $\shD$-modules
\[
 \kappa_p^{\tsc}\colon {j_p^{\tsc}}^*(\boxtimes_{i \in I} \shF_{J_i}) \lsto {j_p^{\tsc}}^*(\shF_J).
\]
In particular, for the surjection $p_{[2]}\colon \{1,2\} \to \{1\}$, we have an isomorphism 
$\kappa_{p_{[2]}}^{\tsc}\colon {j_{p_{[2]}}^{\tsc}}^*(\shF_1^{\boxtimes 2}) \lsto {j_p^{\tsc}}^*(\shF_2)$.
Here and hereafter, we denote $\shF_1 \ceq \shF_{\{1\}}$ and $\shF_2 \ceq \shF_{\{1,2\}}$.

Next, recall the isomorphism 
$\nu_p\colon \Delta_p^*(\shF_J) \sto \shF_I$ of the factorization $\shD$-module $(\shF_I)_I$.
For the surjection $p_{[2]}$, we have $\nu_{p_{[2]}}\colon \Delta^*(\shF_2) \sto \shF_1$,
where $\Delta$ is the (ordinary) diagonal embedding into $C^2$.
Since the superdiagonal embedding $\Delta^{\tsc} \subset C^2$ is a nilpotent extension 
of the ordinary diagonal $\Delta \subset C^2$,
the isomorphism $\nu_{{p_{[2]}}}\colon \Delta^*(\shF_2) \sto \shF_1$ induces a left $\shD$-module morphism
\begin{align}\label{eq:NK:nu}
 \nu_{p_{[2]}}^{\tsc}\colon {\Delta^{\tsc}}^*(\shF_2) \lto \shF_1.
\end{align}

Now we modify the construction of the chiral bracket $\mu$ on $\shA \ceq (\shF_1)^r$ 
\eqref{eq:FS:mu} using \cref{lem:NK:kappa} and \eqref{eq:NK:nu}.
Let us denote $j^{\tsc} \ceq j^{\tsc}_{p_{[2]}}\colon U^{\tsc} \inj C^2$ for simplicity.
Then, by \eqref{eq:SC:jj}, the morphism $\mu_\omega^{\tsc}$ in \eqref{eq:SC:mu-om-sc} is written as 
$\mu_\omega^{\tsc}\colon j^{\tsc}_*{j^{\tsc}}^*(\omega_C^{\boxtimes 2}) \to \Delta^{\tsc}_*\omega_C$.
Also, the left $\shD$-module morphism $\nu_{p_{[2]}}^{\tsc}$ induces a right $\shD$-module morphism
$(\nu^{\tsc})^r\colon (\Delta^{\tsc}_*\omega_C)\otimes\shF_2 \to \Delta^\tsc_*(\omega_C\otimes \shF_1)
=\Delta^{\tsc}_*\shA$.
Now we can replace $\kappa_{p_{[2]}}$, $\mu_\omega$ and $\nu^r$ in \eqref{eq:FS:mu} 
by $\kappa_{p_{[2]}}^{\tsc}$, $\mu_\omega^{\tsc}$ and $(\nu^{\tsc})^r$ to obtain 
\begin{align}\label{eq:NK:mu-sc}
 \mu^{\tsc}\colon \shA^{\boxtimes 2}(\infty\Delta^s) = 
 j^{\tsc}_*{j^{\tsc}}^*\bigl((\omega_C \otimes \shF_1)^{\boxtimes 2}\bigr) 
 \xr[\id \otimes \kappa_{p_{[2]}}^{\tsc}]{\sim}
 \bigl(j^{\tsc}_*{j^{\tsc}}^*(\omega_C^{\boxtimes 2})\bigr) \otimes \shF_2
 \xr[\mu_\omega^{\tsc} \otimes \id]{} (\Delta^{\tsc}_*\omega_C) \otimes \shF_2
 \xr[(\nu^{\tsc})^r]{}
 \Delta^{\tsc}_*\shA.
\end{align}

According to \cite[\S4]{H}, the obtained chiral product $\mu$ and 
a choice of a superconformal coordinate system $Z=(z,\zeta)$ at a point $c \in C$ 
determine the structure of an $N_K=1$ SUSY vertex algebra on the vector superspace $V$ 
which is the fiber of $\shF_1$ at $c$.
By \cref{cor:FS:CdR}, we know that $V$ is isomorphic to the chiral de Rham complex $\Omega_X^{\ch}$,
and the obtained structure of $N_K=1$ SUSY vertex algebra on $V$ 
is consistent with the vertex algebra on $V \cong \Omega_X^{\ch}$.
Thus, we see that the obtained $N_K=1$ SUSY structure is the standard one 
known in the literature (see \cite[\S4]{BHS} and \cite[Example 5.13]{HK}).

We conclude that, for a smooth algebraic variety $X$ and a superconformal curve $(C,\shS)$, 
the factorization structure on the global superloop space $\JC{}X$ naturally induces 
the standard $N_K=1$ SUSY vertex algebra structure on the chiral de Rham complex $\Omega_X^{\ch}$.
This is the confirmation of the comment in \cite[Remark 5.3.4]{KV11}.


\end{document}